\documentclass{amsart}
\issueinfo{00}
{0} 
{0} 
{2007} 

\newtheorem{theorem}{Theorem}[section]

\newtheorem{corollary}[theorem]{Corollary}

\newtheorem{lemma}[theorem]{Lemma}

\theoremstyle{remark}
\newtheorem{remark}[theorem]{Remark}
\numberwithin{equation}{section}

\begin{document}
\title[The foliated structure of contact metric $(\kappa,\mu)$-spaces]{The foliated structure of contact metric $(\kappa,\mu)$-spaces}

\author[B. Cappelletti Montano]{Beniamino Cappelletti Montano}
\address{Dipartimento di Matematica,
Universit\`{a} degli Studi di Bari, Via E. Orabona 4, 70125 Bari,
Italy} \email{cappelletti@dm.uniba.it}



\begin{abstract}
In this note we study the foliated structure of a contact metric
$(\kappa,\mu)$-space. In particular, using the theory of Legendre
foliations, we give a geometric interpretation of the Boeckx's
classification of contact metric $(\kappa,\mu)$-spaces and we find
necessary conditions for a contact manifold to admit a compatible
contact metric $(\kappa,\mu)$-structure. Finally we prove that any
contact metric $(\kappa,\mu)$-space $M$ whose Boeckx invariant $I_M$
is different from $\pm 1$ admits a compatible Sasakian or
Tanaka-Webster parallel structure according to the circumstance that
$|I_M|>1$ or $|I_M|<1$, respectively.
\end{abstract}

\maketitle

\section{Introduction}
A contact metric manifold $(M,\varphi,\xi,\eta,g)$ is called a
contact metric $(\kappa,\mu)$-manifold if the Reeb vector field
$\xi$ belongs to the $\left(\kappa,\mu\right)$-nullity distribution,
i.e. the curvature tensor field satisfies, for all vector fields $X$
and $Y$ on $M$,
\begin{equation}\label{definizione}
R_{X
Y}\xi=\kappa\left(\eta\left(Y\right)X-\eta\left(X\right)Y\right)+\mu\left(\eta\left(Y\right)hX-\eta\left(X\right)hY\right),
\end{equation}
for some real numbers $\kappa$ and $\mu$; here $2h$ denotes the
Lie derivative of $\varphi$ in the direction of $\xi$. This
definition was introduced by Blair,  Kouforgiorgos and Papantoniou
(\cite{BKP-95}) and can be regarded as a generalization both of
the Sasakian condition $R_{X
Y}\xi=\eta\left(Y\right)X-\eta\left(X\right)Y$ and of those
contact metric manifolds verifying $R_{X Y}\xi=0$ which were
studied by D. E. Blair in \cite{blair-1}.

Lately, contact metric $\left(\kappa,\mu\right)$-manifolds have
attracted the attention of many authors and various recent papers
have appeared on this topic (e.g. \cite{Boeckx-08}, \cite{ghosh},
\cite{koufogiorgos}). In fact there are many motivations for
studying $\left(\kappa,\mu\right)$-manifolds: the first is that, in
the non-Sasakian case (that is for $\kappa\neq 1$), the condition
\eqref{definizione} determines the curvature completely; moreover,
while the values of $\kappa$ and $\mu$ may change, the form of
\eqref{definizione} is invariant under $\mathcal D$-homothetic
deformations; finally, a complete classification of contact metric
$\left(\kappa,\mu\right)$-manifolds is known (\cite{Boeckx-00}) and
there are non-trivial examples of such manifolds, the most important
being the unit tangent sphere bundle of a Riemannian manifold of
constant sectional curvature with its usual contact metric
structure.

One of the peculiarities of contact metric $(\kappa,\mu)$-manifolds
is that they give rise to three mutually orthogonal involutive
distributions ${\mathcal D}({\lambda})$, ${\mathcal D}({-\lambda})$
and $\mathbb{R}\xi={\mathcal D}(0)$, corresponding to the
eigenspaces $\lambda$, $-\lambda$ and $0$ of the operator $h$, where
$\lambda=\sqrt{1-\kappa}$. In particular ${\mathcal D}({\lambda})$
and ${\mathcal D}({-\lambda})$ define two transverse Legendre
foliations of $M$ so that any contact metric $(\kappa,\mu)$-manifold
is canonically endowed with a bi-Legendrian structure.  The study of
the bi-Legendrian structure of a contact metric
$(\kappa,\mu)$-manifold was initiated in \cite{Mino-Luigia-07},
where the following characterization of contact metric
$(\kappa,\mu)$-manifolds in terms of Legendre foliations was proven.

\begin{theorem} [\cite{Mino-Luigia-07}]\label{principale0}
Let $(M,\varphi,\xi,\eta,g)$ be a non-Sasakian contact metric
manifold.  Then $M$ is a contact metric
$\left(\kappa,\mu\right)$-manifold if and only if it admits two
mutually orthogonal Legendre distributions $L$ and $Q$ and a
unique linear connection $\bar{\nabla}$ satisfying the following
properties:
\begin{enumerate}
\item[{\rm (i)}] $\bar{\nabla}L\subset L$,\quad $\bar{\nabla} Q\subset Q$,
\item[{\rm (ii)}] $\bar{\nabla}\eta=0$,\quad  $\bar{\nabla}d\eta=0$,\quad  $\bar{\nabla}g=0$,\quad  $\bar{\nabla}\varphi=0$,\quad  $\bar{\nabla}h=0$,
\item[{\rm (iii)}] $\bar{T}\left(X,Y\right)=2d\eta\left(X,Y\right){\xi}$ \quad
for  all $X,Y\in\Gamma({\mathcal{D}})$,\\
$\bar{T}(X,\xi)=[\xi,X_{L}]_{Q}+[\xi,X_{Q}]_{L}$ \quad  for all
$X\in\Gamma(TM)$,
\end{enumerate}
where $\bar{T}$ denotes the torsion tensor field of $\bar{\nabla}$
and $X_L$ and $X_Q$ are, respectively, the projections of $X$ onto
the subbundles $L$ and $Q$ of $TM$. Furthermore, $L$ and $Q$ are
integrable and coincide with the eigenspaces ${\mathcal D}(\lambda)$
and ${\mathcal D}(-\lambda)$ of the operator $h$, and $\bar\nabla$
coincides in fact with the bi-Legendrian connection $\nabla^{bl}$
associated to the bi-Legendrian structure $(L,Q)$ (cf.
\cite{Mino-05}, \cite{Mino-07}).
\end{theorem}

Using the approach of Theorem \ref{principale0}, in
\cite{Mino-Luigia-Tripathi-08} the authors recently were able to
prove the strong result that any invariant submanifold of a
non-Sasakian contact metric $(\kappa,\mu)$-space is totally
geodesic.

In this paper the study of the foliated structure of a contact
metric $(\kappa,\mu)$-space is carried on. \ We start with the
following question, which generalizes the well-known problem of
finding conditions ensuring the existence of Sasakian structures
compatible with a given contact form: let $(M,\eta)$ be a contact
manifold; then does $(M,\eta)$ admit a compatible contact metric
$(\kappa,\mu)$-structure? As a matter of fact, the answer to this
question involves the foliated nature of contact metric
$(\kappa,\mu)$-spaces. In particular, we find necessary
conditions, in terms of bi-Legendrian structures, for a contact
manifold $(M,\eta)$ to admit a compatible contact metric
$(\kappa,\mu)$-structure (cf. Theorem \ref{legendre1} and Theorem
\ref{legendre2}).  Moreover, we interpret the Boeckx
classification \cite{Boeckx-00} of contact metric
$(\kappa,\mu)$-manifolds in terms of the Pang classification
\cite{pang} of Legendre foliations, clarifying the geometric
meaning of the invariant
\begin{equation*}
I_M=\frac{1-\frac{\mu}{2}}{\sqrt{1-\kappa}}
\end{equation*}
which was defined by Boeckx in \cite{Boeckx-00} in a rather obscure
way.

It follows that contact metric $(\kappa,\mu)$-spaces divide into $5$
main classes, according to the behavior of each Legendre foliation
${\mathcal D}(\lambda)$ and ${\mathcal D}(-\lambda)$. We prove that
those classes of contact metric $(\kappa,\mu)$-manifolds such that
$|I_M|\neq 1$ admit a family $(\varphi_{a,b},\xi,\eta,g_{a,b})$ of
compatible contact metric $(\kappa_{a,b},\mu_{a,b})$-structures,
where the constants $\kappa_{a,b}$, $\mu_{a,b}$ are parameterized by
the real numbers $a$ and $b$ satisfying the relation
$ab=\left(2-\mu\right)^2-4\left(1-\kappa\right)$, namely,
\begin{equation*}
\kappa_{a,b}=1-\frac{(a-b)^2}{16}, \ \ \ \mu_{a,b}=2-\frac{a+b}{2}.
\end{equation*}
In particular, we show that, in the case $|I_M|>1$, choosing
$a=b$, the above contact metric
$(\kappa_{a,b},\mu_{a,b})$-structures are in fact Sasakian. Thus,
rather surprisingly, it follows that any contact metric
$(\kappa,\mu)$-manifold such that $|I_M|>1$ admits a compatible
Sasakian structure and hence, under the assumption of compactness,
for each $1\leq p\leq 2n$, the $p$-th Betti number of $M$ is even,
where $2n+1$ is the dimension of the manifold. At the knowledge of
the author, the last one is the first topological obstruction for
contact metric $(\kappa,\mu)$-manifolds known at the moment.
Whereas, if $|I_M|<1$, choosing $a=-b$, we obtain a family of
Tanaka-Webster parallel structures, i.e. contact metric structures
whose Tanaka-Webster connection preserves the Tanaka-Webster
torsion and the Tanaka-Webster curvature (\cite{Boeckx-08}).

Finally, we show that those contact metric manifolds with $|I_M|=1$
admit a family $(\varphi_{c},\xi,\eta,g_{c})$ of compatible contact
metric $(\kappa_c,\mu_c)$-structures, with
\begin{equation*}
\kappa_{c}=1-\frac{c^2}{16}, \ \ \
\mu_{c}=2\left(1-\frac{c}{4}\right),
\end{equation*}
where $c$ varies in the interval $(0,4]$ in the case $I_M=1$ and
$[-4,0)$ in the case $I_M=-1$.

\section{Preliminaries}

\subsection{Contact geometry}

A \emph{contact manifold} is a $(2n+1)$-dimensional smooth manifold
$M$ which carries a $1$-form $\eta$, called \emph{contact form},
satisfying $\eta\wedge\left(d\eta\right)^n\neq 0$ everywhere on $M$.
It is well known that given $\eta$ there exists a unique vector
field $\xi$, called \emph{Reeb vector field}, such that
$i_{\xi}\eta=1$ and $i_{\xi}d\eta=0$. In the sequel we will denote
by $\mathcal D$ the $2n$-dimensional distribution defined by
$\ker\left(\eta\right)$, called the \emph{contact distribution}. It
is easy to see that the Reeb vector field is an infinitesimal
automorphism with respect to the contact distribution and  the
tangent bundle of $M$ splits as the direct sum $TM=\mathcal
D\oplus\mathbb{R}\xi$.

It is well known that any contact manifold $(M,\eta)$ admits a
Riemannian metric $g$ and a $(1,1)$-tensor field $\varphi$ such
that
\begin{equation}\label{acca}
\varphi ^{2}=-I+\eta \otimes \xi, \ \ d\eta \left(
X,Y\right)=g\left( X,\varphi Y\right), \ \ g(\varphi X,\varphi
Y)=g(X,Y)-\eta (X)\eta (Y)
\end{equation}
for all $X,Y\in \Gamma \left( TM\right)$, from which it follows that
$\varphi\xi=0$, $\eta\circ\varphi=0$ and $\eta=g(\cdot,\xi)$. The
structure $\left(\varphi, \xi, \eta, g\right)$ is called a
\emph{contact metric structure} and the manifold $M$ endowed with
such a structure is said to be a {\em contact metric manifold}.  In
a contact metric manifold $M$, the $\left( 1,1\right) $-tensor field
$h:=\frac 12 \mathcal L_\xi \varphi$ is symmetric and satisfies
\begin{equation}
h\xi =0,\;\; \eta\circ h=0,\;\; h\varphi +\varphi h=0,\;\;\nabla\xi
=-\varphi -\varphi h,\;\;  {\rm tr}(h)={\rm tr}(\varphi h)=0,
\label{eq-contact-7}
\end{equation}
where $\nabla$ is the Levi Civita connection of $(M,g)$. The tensor
field $h$ vanishes identically if and only if the Reeb vector field
is Killing, and in this case the contact metric manifold in question
is said to be \emph{K-contact}.

Moreover, in any contact metric manifold one can consider the tensor
field $N_{\varphi}$, defined by
\begin{equation*}
N_{\varphi}(X,Y):=\varphi^2[X,Y]+[\varphi X,\varphi
Y]-\varphi[\varphi X,Y]-\varphi[X,\varphi Y]+2d\eta(X,Y)\xi,
\end{equation*}
for all $X,Y\in\Gamma(TM)$. The tensor field $N_\varphi$ satisfies
the following formula, which will turn out very useful in the
sequel,
\begin{equation}\label{formulenijenhuis1}
\varphi N_\varphi(X,Y)+N_\varphi(\varphi X,Y)=2\eta(X)h Y,
\end{equation}
for all $X,Y\in\Gamma(TM)$, from which, in particular, it follows
that
\begin{equation}\label{formulenijenhuis2}
\eta(N_{\varphi}(\varphi X,Y))=0.
\end{equation}
A contact metric manifold such that $N_\varphi$ vanishes identically
is said to be \emph{Sasakian}. In terms of the covariant derivative
the Sasakian condition can be expressed by the following formula
\begin{equation}\label{condizionesasaki}
(\nabla_{X}\varphi)Y=g(X,Y)\xi-\eta(Y)X,
\end{equation}
whereas, in term of the curvature tensor field, the Sasakian
condition is
\begin{equation*}
R_{XY}\xi=\eta(Y)X-\eta(X)Y.
\end{equation*}
Any Sasakian manifold is \emph{K}-contact and in dimension $3$ also
the converse holds (see \cite{Blair-02} for more details).

A recent generalization of Sasakian manifolds is the notion of
\emph{contact metric $(\kappa,\mu)$-manifolds} (\cite{BKP-95}). Let
$(M,\varphi,\xi,\eta,g)$ be a contact metric manifold. If  the
curvature tensor field of the Levi Civita connection satisfies
\begin{equation}\label{eq-km}
R_{X
Y}\xi=\kappa\left(\eta\left(Y\right)X-\eta\left(X\right)Y\right)+\mu\left(\eta\left(Y\right)hX-\eta\left(X\right)hY\right),
\end{equation}
for some $\kappa,\mu\in\mathbb{R}$, we say that
$(M,\varphi,\xi,\eta,g)$  is a \emph{contact metric
$(\kappa,\mu)$-manifold} (or that $\xi$ belongs to the
$(\kappa,\mu)$-nullity distribution). This definition was introduced
and deeply studied by Blair, Koufogiorgos and Papantoniou in
\cite{BKP-95}. Among other things, the authors proved the following
result.

\begin{theorem}[\cite{BKP-95}]\label{teoremagreci}
Let $\left(M,\varphi,\xi,\eta,g\right)$ be a contact metric
$(\kappa,\mu)$-manifold. Then necessarily $\kappa\leq 1$. Moreover,
if $\kappa=1$ then $h=0$ and $\left(M,\varphi,\xi,\eta,g\right)$ is
 Sasakian; if $\kappa<1$, the contact metric structure is not
Sasakian and $M$ admits three mutually orthogonal integrable
distributions ${\mathcal D}(0)=\mathbb{R}\xi$, ${\mathcal
D}(\lambda)$ and ${\mathcal D}(-\lambda)$ corresponding to the
eigenspaces of $h$, where $\lambda=\sqrt{1-\kappa}$.
\end{theorem}
Given a non-Sasakian contact metric $(\kappa,\mu)$-manifold $M$,
Boeckx \cite{Boeckx-00} proved that the number
\begin{equation*}
I_{M}:=\frac{1-\frac{\mu}{2}}{\sqrt{1-\kappa}},
\end{equation*}
is an invariant of the contact metric $(\kappa,\mu)$-structure, and
he demonstrated that two non-Sasakian contact metric
$(\kappa,\mu)$-manifolds $(M_1,\varphi_1,\xi_1,\eta_1,g_1)$ and
$(M_2,\varphi_2,\xi_2,\eta_2,g_2)$ are locally isometric as contact
metric manifolds if and only if $I_{M_1}=I_{M_2}$. Then the
invariant $I_M$ has been used by Boeckx for giving a full
classification of contact metric $(\kappa,\mu)$-spaces.

The standard example of contact metric $(\kappa,\mu)$-manifolds is
given by the tangent sphere bundle $T_{1}N$ of a manifold of
constant curvature $c$ endowed with its standard contact metric
structure. In this case $\kappa=c(2-c)$, $\mu=-2c$ and
$I_{T_{1}N}=\frac{1+c}{|1-c|}$. Therefore as $c$ varies over the
reals, $I_{T_{1}N}$ takes on every value strictly greater than $-1$.
Moreover one can easily find that $I_{T_{1}N}<1$ if and only if
$c<0$.

\subsection{Legendre foliations}\label{preliminari} Let $(M,\eta)$ be
a $(2n+1)$-dimensional contact manifold. Notice that the condition
$\eta\wedge(d\eta)^n\neq 0$ implies that the contact distribution is
never integrable. One can prove that in fact the maximal dimension
of an integrable subbundle of $\mathcal D$ is $n$. This motivates
the following definition. A \emph{Legendre distribution} on a
contact manifold $(M,\eta)$ is an $n$-dimensional subbundle $L$ of
the contact distribution such that $d\eta\left(X,X'\right)=0$ for
all $X,X'\in\Gamma\left(L\right)$. Then by a \emph{Legendre
foliations} of $(M,\eta)$ we mean a foliation $\mathcal F$ of $M$
whose tangent bundle $L=T{\mathcal F}$ is a Legendre distribution,
according to the above definition.

Legendre foliations have been extensively investigated in recent
years from various points of views. In particular  Pang
(\cite{pang}) provided a classification of Legendre foliations by
means of a bilinear symmetric form $\Pi_{\mathcal F}$ on the tangent
bundle of the foliation ${\mathcal F}$, defined by
\begin{equation*}
\Pi_{\mathcal F}\left(X,X'\right)=-\left({\mathcal L}_{X}{\mathcal
L}_{X'}\eta\right)\left(\xi\right)=2d\eta([\xi,X],X').
\end{equation*}
He called a Legendre foliation \emph{non-degenerate},
\emph{degenerate} or \emph{flat} according to the circumstance that
the bilinear form $\Pi_{\mathcal F}$ is non-degenerate, degenerate
or vanishes identically, respectively. In terms of Lie brackets, the
flat condition is equivalent to the requirement that
$[\xi,X]\in\Gamma(T{\mathcal F})$ for all $X\in\Gamma(T{\mathcal
F})$.  Two interesting subclasses of non-degenerate Legendre
foliations are given by those for which $\Pi_{\mathcal F}$ is
positive definite and negative definite; we then speak of
\emph{positive definite} and \emph{negative definite Legendre
foliations}, respectively.

For \ a \ non-degenerate \ Legendre \ foliation \ $\mathcal F$, \
Libermann \ (\cite{libermann}) \ defined \ a \ linear \ map
$\Lambda_{\mathcal F}:TM\longrightarrow T{\mathcal F}$, whose
kernel is ${T\mathcal F}\oplus\mathbb{R}\xi$, such that
\begin{equation}\label{lambda}
\Pi_{\mathcal F}(\Lambda_{\mathcal F} Z,X)=d\eta(Z,X)
\end{equation}
for any $Z\in\Gamma(TM)$, $X\in\Gamma(T{\mathcal F})$. The operator
$\Lambda_{\mathcal F}$ is surjective,  verifies $(\Lambda_{\mathcal
F})^2=0$ and
\begin{equation}\label{proplambda}
\Lambda_{\mathcal F}[\xi,X]=\frac{1}{2}X
\end{equation}
for all $X\in\Gamma(T{\mathcal F})$.  Then we can extend
$\Pi_{\mathcal F}$ to a symmetric bilinear form on $TM$ by putting
\begin{equation*}
\overline\Pi_{\mathcal F}(Z,Z'):=\left\{
                                   \begin{array}{ll}
                                     \Pi_{\mathcal F}(Z,Z') & \hbox{if $Z,Z'\in\Gamma(T{\mathcal F})$} \\
                                     \Pi_{\mathcal F}(\Lambda_{\mathcal
F} Z,\Lambda_{\mathcal F} Z'), & \hbox{otherwise.}
                                   \end{array}
                                 \right.
\end{equation*}

If  $(M,\eta)$   admits  two  transversal  Legendre  distributions
$L_1$  and  $L_2$,  we  say  that $(M,\eta,L_1,L_2)$ is an
\emph{almost bi-Legendrian manifold}. Thus, in particular, the
tangent bundle of $M$ splits up as the direct sum $TM=L_1\oplus
L_2\oplus\mathbb{R}\xi$. When both $L_1$ and $L_2$ are integrable we
speak of \emph{bi-Legendrian manifold}. An (almost) bi-Legendrian
manifold is said to be flat, degenerate or non-degenerate if and
only if both the Legendre distributions are flat, degenerate or
non-degenerate, respectively.   Any contact manifold $(M,\eta)$
endowed with a Legendre distribution $L$ admits a canonical almost
bi-Legendrian structure. Indeed let $(\varphi,\xi,\eta,g)$ be a
compatible contact metric structure. Then by the relation
$d\eta(\phi\cdot,\phi\cdot)=d\eta$ it easily follows that $Q:=\phi
L$ is a Legendre distribution on $M$ which is $g$-orthogonal to $L$.
$Q$ is usually called the \emph{conjugate Legendre distribution} of
$L$ and in general is not integrable, even if $L$ is.

In \cite{Mino-05} (see also \cite{Mino-07}) a canonical connection,
which plays an important role in the study of almost bi-Legendrian
manifolds, has been introduced:

\begin{theorem}[\cite{Mino-05}]\label{biconnection}
Let $(M,\eta,L_1,L_2)$ be an almost bi-Legendrian manifold. There
exists a unique connection ${\nabla}^{bl}$ such that
\begin{enumerate}
  \item[(i)] ${\nabla}^{bl} L_1\subset L_1$, \ ${\nabla}^{bl} L_2\subset L_2$,
  \item[(ii)]  ${\nabla}^{bl}\xi=0$, \ ${\nabla}^{bl} d\eta=0$,
  \item[(iii)] ${T}^{bl}\left(X,Y\right)=2d\eta\left(X,Y\right){\xi}$ \ for all
  $X\in\Gamma(L_1)$, $Y\in\Gamma(L_2)$,\\
${T}^{bl}\left(X,\xi\right)=[\xi,X_{L_1}]_{L_2}+[\xi,X_{L_2}]_{L_1}$
  \ for all   $X\in\Gamma\left(TM\right)$,
\end{enumerate}
where ${T}^{bl}$ denotes the torsion tensor field of ${\nabla}^{bl}$
and $X_{L_1}$ and $X_{L_2}$ the projections of $X$ onto the
subbundles $L_1$ and $L_2$ of $TM$, respectively.
\end{theorem}

Such  a  connection  is  called  the  \emph{bi-Legendrian
connection}  of   the  almost bi-Legendrian manifold
$(M,\eta,L_1,L_2)$. We recall also the complete expression of the
torsion tensor field of $\nabla^{bl}$,
\begin{align}\label{torsione}
T^{bl}(X,Y)&=-[X_{L_1},Y_{L_1}]_{L_2\oplus\mathbb{R}\xi}-[X_{L_2},Y_{L_2}]_{L_1\oplus\mathbb{R}\xi}+2d\eta(X,Y)\xi\nonumber\\
&\quad+\eta(Y)\left([\xi,X_{L_1}]_{L_2}+[\xi,X_{L_2}]_{L_1}\right)-\eta(X)\left([\xi,Y_{L_1}]_{L_2}+[\xi,Y_{L_2}]_{L_1}\right).
\end{align}

\section{The main results}\label{primasezione}

By Theorem \ref{teoremagreci} it follows that any non-Sasakian
contact metric $(\kappa,\mu)$-manifold is endowed with a canonical
 bi-Legendrian structure given by the mutually
orthogonal integrable distributions ${\mathcal D}(\lambda)$ and
${\mathcal D}(-\lambda)$. Therefore we can classify non-Sasakian
contact metric $(\kappa,\mu)$-manifolds by using the aforementioned
Pang's classification of Legendre foliations based on the behavior
of the invariants $\Pi_{\mathcal D(\lambda)}$ and $\Pi_{\mathcal
D(-\lambda)}$. The explicit expression of the invariants of the
Legendre foliations defined by ${\mathcal D}(\lambda)$ and
${\mathcal D}(-\lambda)$ was found in in \cite{Mino-Luigia-07}:
\begin{gather}
\Pi_{{\mathcal
D}(\lambda)}=\frac{\left(\lambda+1\right)^2-\kappa-\mu\lambda}{\lambda}g|_{{\mathcal
D}(\lambda)\times{\mathcal
D}(\lambda)}=\left(2\sqrt{1-\kappa}-\mu+2\right)g|_{{\mathcal
D}(\lambda)\times{\mathcal D}(\lambda)},\label{invariante1}\\
\Pi_{{\mathcal
D}(-\lambda)}=\frac{-\left(\lambda-1\right)^2+\kappa-\mu\lambda}{\lambda}g|_{{\mathcal
D}(-\lambda)\times{\mathcal
D}(-\lambda)}=\left(-2\sqrt{1-\kappa}-\mu+2\right)g|_{{\mathcal
D}(-\lambda)\times{\mathcal D}(-\lambda)}\label{invariante2}.
\end{gather}
Using \eqref{invariante1}--\eqref{invariante2} we can classify
non-Sasakian contact metric $(\kappa,\mu)$-manifolds as follows.

\begin{theorem}\label{classificazione}
Let $(M,\varphi,\xi,\eta,g)$ be a non-Sasakian contact metric
$(\kappa,\mu)$-manifold. Then the bi-Legendrian structure
$({\mathcal D}(\lambda),{\mathcal D}(-\lambda))$ associated to
$(M,\varphi,\xi,\eta,g)$ is non-flat. More precisely, only one among
the following cases occurs:
\begin{enumerate}
  \item[(I)] both ${\mathcal D}(\lambda)$ and  ${\mathcal
D}(-\lambda)$ are positive definite;
  \item[(II)] ${\mathcal D}(\lambda)$ is positive definite and  ${\mathcal
D}(-\lambda)$ is negative definite;
\item[(III)] both ${\mathcal D}(\lambda)$ and  ${\mathcal
D}(-\lambda)$ are negative definite;
\item[(IV)]  ${\mathcal D}(\lambda)$ is positive definite and  ${\mathcal
D}(-\lambda)$ is flat;
\item[(V)] ${\mathcal D}(\lambda)$ is flat and ${\mathcal
D}(-\lambda)$ is negative definite.
\end{enumerate}
Furthermore, $M$ belongs to the class (I), (II), (III), (IV), (V) if
and only if $I_M>1$, $-1<I_M<1$,  $I_M<-1$, $I_M=1$, $I_M=-1$,
respectively.
\end{theorem}
\begin{proof}
By \eqref{invariante1}--\eqref{invariante2} we have immediately
that ${\mathcal D}(\lambda)$ and ${\mathcal D}(-\lambda)$ are
either positive definite or positive negative or flat, depending
on the sign of the functions
$f_1(\kappa,\mu)=2\sqrt{1-\kappa}-\mu+2$ and
$f_2\left(\kappa,\mu\right)=-2\sqrt{1-\kappa}-\mu+2$. Since
$f_1\left(\kappa,\mu\right)$ and $f_2\left(\kappa,\mu\right)$ both
vanish if and only if $\kappa=1$, the bi-Legendrian structure
$({\mathcal D}(\lambda),{\mathcal D}(-\lambda))$ turns out to be
non-flat. Moreover, one easily finds that $f_1(\kappa,\mu)>0$ if
and only if $I_M>-1$ and $f_2(\kappa,\mu)>0$ if and only if
$I_M>1$. Consequently, taking into account
\eqref{invariante1}--\eqref{invariante2}, the cases
$\Pi_{{\mathcal D}(\lambda)}$ negative definite and
$\Pi_{{\mathcal D}(-\lambda)}$ positive definite, $\Pi_{{\mathcal
D}(\lambda)}=0$ and $\Pi_{{\mathcal D}(-\lambda)}$ positive
definite, $\Pi_{{\mathcal D}(\lambda)}$ negative definite and
$\Pi_{{\mathcal D}(-\lambda)}=0$ can not occur, and the remaining
combinations of all possible signs of $f_1(\kappa,\mu)$ and of
$f_2(\kappa,\mu)$ give the claimed assertion.
\end{proof}

Using Theorem \ref{classificazione} we are able to study the
following interesting problem. It is a well-known question in
contact geometry whether, given a contact manifold $(M,\eta)$, there
exists a Sasakian structure on $M$ compatible with the contact form
$\eta$. Now we generalize this problem and we ask whether, given a
contact manifold $(M,\eta)$, there exists a compatible contact
metric structure $(\varphi,\xi,\eta,g)$ such that
$(M,\varphi,\xi,\eta,g)$ is a contact metric
$(\kappa,\mu)$-manifold. In order to answer this question we need to
recall the following lemma proven in \cite{Mino-07}.

\begin{lemma}[\cite{Mino-07}]\label{lemmarocky}
Let $(M,\varphi,\xi,\eta,g)$ be a contact metric manifold endowed
with a Legendre distribution $L$. Let $Q:=\varphi L$ be the
conjugate Legendre distribution of $L$ and $\nabla^{bl}$ the
bi-Legendrian connection associated to $\left(L,Q\right)$. Then the
following statements are equivalent:
\begin{enumerate}
    \item[(i)] $\nabla^{bl} g=0$.
    \item[(ii)] $\nabla^{bl}\varphi=0$.
    \item[(iii)] $\nabla^{bl}_{X}X'=-\left(\varphi\left[X,\varphi
    X'\right]\right)_L$ for all $X,X'\in\Gamma\left(L\right)$, $\nabla^{bl}_{Y}Y'=-\left(\varphi\left[Y,\varphi
    Y'\right]\right)_Q$ for all $Y,Y'\in\Gamma\left(Q\right)$ and
    the tensor field $h$ maps the subbundle $L$
    onto $L$ and the subbundle $Q$ onto $Q$.
    \item[(iv)] $g$ is a bundle-like metric with respect
both to the distribution $L\oplus \mathbb{R}\xi$ and to the
distribution $Q\oplus \mathbb{R}\xi$.
\end{enumerate}
Furthermore, assuming $L$ and $Q$ integrable, (i)--(iv) are
equivalent to the total geodesicity (with respect to the Levi Civita
connection of $g$) of the Legendre foliations defined by $L$ and
$Q$.
\end{lemma}

\begin{theorem}\label{legendre1}
Let $(M,\eta)$ be a contact manifold endowed with a  bi-Legendrian
structure $({\mathcal F}_1,{\mathcal F}_2)$ such that
$\nabla^{bl}\Pi_{{\mathcal F}_1}=\nabla^{bl}\Pi_{{\mathcal F}_2}=0$.
Assume that one of the following conditions holds
\begin{itemize}
  \item[(I)] ${\mathcal F}_1$ and ${\mathcal F}_2$ are positive
definite and there exist two positive
 numbers $a$ and $b$ such that $\overline\Pi_{{\mathcal
F}_1}=ab\overline\Pi_{{\mathcal F}_2}$ on $T{\mathcal F}_1$ and
$\overline\Pi_{{\mathcal F}_2}=ab\overline\Pi_{{\mathcal F}_1}$ on
$T{\mathcal F}_2$,
  \item[(II)] ${\mathcal F}_1$ is positive definite,  ${\mathcal F}_2$
is negative definite and there exist $a>0$ and $b<0$ such that
$\overline\Pi_{{\mathcal F}_1}=ab\overline\Pi_{{\mathcal F}_2}$ on
$T{\mathcal F}_1$ and $\overline\Pi_{{\mathcal
F}_2}=ab\overline\Pi_{{\mathcal F}_1}$ on $T{\mathcal F}_2$,
  \item[(III)]  ${\mathcal F}_1$ and ${\mathcal F}_2$ are negative
definite and there exist two negative
 numbers $a$ and $b$ such that $\overline\Pi_{{\mathcal
F}_1}=ab\overline\Pi_{{\mathcal F}_2}$ on $T{\mathcal F}_1$ and
$\overline\Pi_{{\mathcal F}_2}=ab\overline\Pi_{{\mathcal F}_1}$ on
$T{\mathcal F}_2$.
\end{itemize}
Then $(M,\eta)$ admits a  compatible contact metric structure
$(\varphi,\xi,\eta,g)$ such that
\begin{enumerate}
  \item[(i)] if $a=b$, $(M,\varphi,\xi,\eta,g)$ is a Sasakian
manifold;
  \item[(ii)] if $a\neq b$, $(M,\varphi,\xi,\eta,g)$ is a contact metric
$(\kappa,\mu)$-manifold, whose  associated bi-Legendrian structure
is $({\mathcal F}_1,{\mathcal F}_2)$, where
\begin{equation}\label{costanti0}
\kappa=1-\frac{(a-b)^2}{16}, \ \ \ \mu=2-\frac{a+b}{2}.
\end{equation}
\end{enumerate}
\end{theorem}
\begin{proof}
We  consider,  for  each   Legendre  foliation  ${\mathcal F}_1$ and
 ${\mathcal F}_2$,  the  Libermann  operators $\Lambda_{{\mathcal
F}_1}:TM\longrightarrow T{\mathcal F}_1$ and $\Lambda_{{\mathcal
F}_2}:TM\longrightarrow T{\mathcal F}_2$ defined by \eqref{lambda}.
Then we set
\begin{equation}\label{costruzionemetrica}
g|_{T{\mathcal F}_1\times T{\mathcal
F}_1}:=\frac{1}{a}\Pi_{{\mathcal F}_1}, \ \ g|_{T{\mathcal
F}_2\times T{\mathcal F}_2}:=\frac{1}{b}\Pi_{{\mathcal F}_2}, \ \
g:=\eta\otimes\eta \ \textrm{elsewhere.}
\end{equation}
That $g$ is a Riemannian metric follows from the fact that the
bilinear map $\Pi_{{\mathcal F}_1}$ and $\Pi_{{\mathcal F}_2}$ are
symmetric and, by the assumptions (I)--(III), they are positive or
negative definite according to the signs of $a$ and $b$,
respectively, in such a way that the bilinear forms
$\frac{1}{a}\Pi_{{\mathcal F}_1}$ and $\frac{1}{b}\Pi_{{\mathcal
F}_2}$  are always positive definite. In particular by
\eqref{costruzionemetrica} we have that ${\mathcal F}_1={{\mathcal
F}_2}^{\perp}\cap\mathcal D$ and ${\mathcal F}_2={{\mathcal
F}_1}^{\perp}\cap\mathcal D$.  Next let us define a tensor field
$\varphi$ by
\begin{equation}\label{struttura}
\varphi Z:=\left\{
             \begin{array}{ll}
               -b\Lambda_{{\mathcal F}_2} Z, & \hbox{if $Z\in\Gamma(T{\mathcal F}_1)$,} \\
               -a\Lambda_{{\mathcal F}_1} Z, & \hbox{if $Z\in\Gamma(T{\mathcal F}_2)$,} \\
               0, & \hbox{if $Z\in\Gamma(\mathbb{R}\xi)$.}
             \end{array}
           \right.
\end{equation}
Notice that, by definition, $\varphi$ maps $T{\mathcal F}_1$ onto
$T{\mathcal F}_2$ and $T{\mathcal F}_2$ onto $T{\mathcal F}_1$.
Moreover, for any $X,X'\in\Gamma(T{\mathcal F}_1)$,
\begin{gather*}
\Pi_{{\mathcal F}_1}(\varphi^2 X,X')=ab\Pi_{{\mathcal
F}_1}(\Lambda_{{\mathcal F}_1}\Lambda_{{\mathcal F}_2} X,X')=ab
d\eta(\Lambda_{{\mathcal F}_2} X,X')=-ab d\eta(X',\Lambda_{{\mathcal
F}_2} X)\\
=-ab \Pi_{{\mathcal F}_2}(\Lambda_{{\mathcal F}_2}
X',\Lambda_{{\mathcal F}_2} X)=-ab\overline\Pi_{{\mathcal
F}_2}(X,X')=-\frac{ab}{ab}\overline\Pi_{{\mathcal
F}_1}(X,X')=-\Pi_{{\mathcal F}_1}(X,X')
\end{gather*}
from which it follows that $\varphi^2 X=-X$. Analogously one can
prove that $\varphi^2 Y=-Y$ for all $Y\in\Gamma(T{\mathcal F}_2)$.
Thus $\varphi^2=-I+\eta\otimes\xi$. We prove that
$(\varphi,\xi,\eta,g)$ is in fact a contact metric structure.
Indeed, for all $X,X'\in\Gamma(T{\mathcal F}_1)$
\begin{gather*}
g(\varphi X,\varphi X')=b^2 g(\Lambda_{{\mathcal F}_2}
X,\Lambda_{{\mathcal F}_2} X')=b\Pi_{{\mathcal
F}_2}(\Lambda_{{\mathcal F}_2} X,\Lambda_{{\mathcal F}_2}
X')\\
=b\overline\Pi_{{\mathcal F}_2}(X,X')=\frac{1}{a}\Pi_{{\mathcal
F}_1}(X,X')=g(X,X').
\end{gather*}
Analogously one has $g(\varphi Y,\varphi Y')=g(Y,Y')$ for all
$Y,Y'\in\Gamma(T{\mathcal F_2})$, so that we can conclude that
$g(\varphi\cdot,\varphi\cdot)=g(\cdot,\cdot)-\eta\otimes\eta$.
Furthermore, for all $X\in\Gamma(T{\mathcal F}_1)$ and
$Y\in\Gamma(T{\mathcal F}_2)$ we have
\begin{gather*}
g(X,\varphi Y)=\frac{1}{a}\Pi_{{\mathcal F}_1}(X,\varphi
Y)=-\Pi_{{\mathcal F}_1}(X,\Lambda_{{\mathcal F}_1}
Y)\\
=-\Pi_{{\mathcal F}_1}(\Lambda_{{\mathcal F}_1}
Y,X)=-d\eta(Y,X)=d\eta(X,Y)
\end{gather*}
and, in the same way, $g(Y,\varphi X)=d\eta(Y,X)$. Moreover, since
${\mathcal F}_1$ and ${\mathcal F}_2$ are mutually orthogonal with
respect to $g$ and they are Legendre foliations, we have
$d\eta(X,X')=0=g(X,\varphi X')$ and $d\eta(Y,Y')=0=g(Y,\varphi Y')$
for all $X,X'\in\Gamma(T{\mathcal F}_1)$ and for all
$Y,Y'\in\Gamma(T{\mathcal F}_2)$. Therefore
$d\eta=g(\cdot,\varphi\cdot)$ and $(\varphi,\xi,\eta,g)$ is contact
metric structure. Notice that ${\mathcal F}_1$ and ${\mathcal F}_2$
are conjugate Legendre foliations with respect to
$(\varphi,\xi,\eta,g)$, since $\varphi(T{\mathcal F}_1)=T{\mathcal
F}_2$ and $\varphi(T{\mathcal F}_2)=T{\mathcal F}_1$.  Now, since
$\nabla^{bl}\Pi_{{\mathcal F}_1}=\nabla^{bl}\Pi_{{\mathcal F}_2}=0$,
we have that the bi-Legendrian connection preserves the Riemannian
metric $g$ and this, by Lemma \ref{lemmarocky}, implies that
$\nabla^{bl}\varphi=0$ and $h:=\frac{1}{2}{\mathcal L}_{\xi}\varphi$
preserves the foliations ${\mathcal F}_1$ and ${\mathcal F}_2$.
Then, as $\ker(\Lambda_{{\mathcal F}_1})=T{\mathcal
F}_1\oplus\mathbb{R}\xi$ and by \eqref{proplambda} we have, for any
$X\in\Gamma(T{\mathcal F}_1)$, $\varphi([\xi,X]_{T{\mathcal
F}_2})=-a\Lambda_{{\mathcal F}_1}([\xi,X]_{T{\mathcal
F}_2})=-a\Lambda_{{\mathcal F}_1}[\xi,X]=-\frac{a}{2}X$, hence
\begin{equation}\label{formulaphi1}
\varphi X=\frac{2}{a}[\xi,X]_{T{\mathcal F}_2}.
\end{equation}
Analogously, one can prove that
\begin{equation}\label{formulaphi2}
\varphi Y=\frac{2}{b}[\xi,Y]_{T{\mathcal F}_1}
\end{equation}
for all $Y\in\Gamma(T{\mathcal F}_2)$. Thus for any
$X\in\Gamma(T{\mathcal F}_1)$, $2h X=[\xi,\varphi
X]-\varphi[\xi,X]=[\xi,\varphi X]_{T{\mathcal F}_1}+[\xi,\varphi
X]_{T{\mathcal F}_2}-\varphi([\xi,X]_{T{\mathcal
F}_1})-\varphi([\xi,X]_{T{\mathcal F}_2})$, from which, as
$h(T{\mathcal F}_1)\subset T{\mathcal F}_1$, it follows that $2 h X
- [\xi,\varphi X]_{T{\mathcal F}_1} + \varphi([\xi,X]_{T{\mathcal
F}_2}) =[\xi,\varphi X]_{T{\mathcal
F}_2}-\varphi([\xi,X]_{T{\mathcal F}_1})=0$. Hence, using
\eqref{formulaphi1}--\eqref{formulaphi2},
\begin{equation}\label{formulah1}
h X = \frac{1}{2}\left([\xi,\varphi X]_{T{\mathcal
F}_1}-\varphi([\xi,X]_{T{\mathcal
F}_2})\right)=\frac{1}{2}\left(-\frac{b}{2}+\frac{a}{2}\right)X=\frac{a-b}{4}X.
\end{equation}
In the same way one has, for any $Y\in\Gamma(T{\mathcal F}_2)$,
\begin{equation}\label{formulah2}
h Y = \frac{1}{2}\left([\xi,\varphi Y]_{T{\mathcal
F}_2}-\varphi([\xi,Y]_{T{\mathcal
F}_1})\right)=\frac{1}{2}\left(-\frac{a}{2}+\frac{b}{2}\right)Y=-\frac{a-b}{4}Y.
\end{equation}
We then distinguish the cases $a\neq b$ and $a=b$. In the first
case, assuming for instance $a>b$, the manifold is not
\emph{K}-contact and ${\mathcal F}_1$ and ${\mathcal F}_2$ are the
eigenspaces of the operator $h$ corresponding to the eigenvalues
$\lambda=\frac{a-b}{4}$ and $-\lambda$, respectively. Therefore
$\nabla^{bl}h=0$ and so $(M,\varphi,\xi,\eta,g)$ fulfils all the
conditions required by Theorem \ref{principale0} and we conclude
that it is a contact metric $(\kappa,\mu)$-manifold. Comparing
\eqref{costruzionemetrica} with
\eqref{invariante1}--\eqref{invariante2} we obtain the linear system
$2\lambda-\mu+2=a, \ -2\lambda-\mu+2=b$ which admits the unique
solution $\lambda=\frac{a-b}{4}$, $\mu=2-\frac{a+b}{2}$. Hence
$\kappa=1-\lambda^2=1-\frac{(a-b)^2}{16}$. Now we consider the case
$a=b$. By \eqref{formulah1}--\eqref{formulah2} we have that $h=0$.
Due to (iii) of Lemma \ref{lemmarocky} and using \eqref{torsione} we
have for all $X,X'\in\Gamma(T\mathcal F)$
\begin{align*}
(N_{\varphi}(X,X'))_{T{\mathcal F}_1}&=-[X,X']-(\varphi[\varphi
X,X'])_{T{\mathcal F}_1}-(\varphi[X,\varphi X'])_{T{\mathcal F}_1}\\
&=-[X,X']-\nabla^{bl}_{X'}X+\nabla^{bl}_{X}X'\\
&=T^{bl}(X,X')\\
&=-[X,X']_{T{\mathcal F}_2\oplus\mathbb{R}\xi}=0,
\end{align*}
because of the integrability of ${\mathcal F}_1$. Analogously,
$(N_{\varphi}(Y,Y'))_{T{\mathcal F}_2}=0$ for all
$Y,Y'\in\Gamma(T{\mathcal F}_2)$. Now, for all
$X,X'\in\Gamma(T{\mathcal F}_1)$,
\begin{align*}
N_{\varphi}(\varphi X,\varphi X')&=-[\varphi X,\varphi
X']+[\varphi^2 X,\varphi^2 X']-\varphi[\varphi^2 X,\varphi
X']-\varphi[\varphi X,\varphi^2 X']\\
&=-[\varphi X,\varphi X']+[X,X']+\varphi[X,\varphi
X']+\varphi[\varphi X,X']\\
&=-N_{\varphi}(X,X'),
\end{align*}
hence \ $(N_\varphi(X,X'))_{T{\mathcal F}_2}=-(N_\varphi(\varphi
X,\varphi X'))_{T{\mathcal F}_2}=0$. \ Since \ by \
\eqref{formulenijenhuis2} \
$g(N_{\varphi}(X,X'),\xi)=\eta(N_{\varphi}(X,X'))=0$,
$N_{\varphi}(X,X')$ has zero component also in the direction of
$\xi$ and we conclude that $N_{\varphi}(X,X')\equiv 0$. In the same
way one can show that $N_{\varphi}(Y,Y')\equiv 0$ for all
$Y,Y'\in\Gamma(T{\mathcal F}_2)$. Moreover,
\eqref{formulenijenhuis1} implies that $N_{\varphi}(X,Y)=0$ for all
$X\in\Gamma(T{\mathcal F}_1)$ and $Y\in\Gamma(T{\mathcal F}_2)$.
Finally, directly by the definition of $N_\varphi$ we have
$\eta(N_{\varphi}(Z,\xi))=0$ for all $Z\in\Gamma({\mathcal D})$, and
from \eqref{formulenijenhuis1} it follows  that $\varphi
N_\varphi(Z,\xi)=0$. Hence
$N_\varphi(Z,\xi)\in\ker(\eta)\cap\ker(\varphi)=\{0\}$. Thus the
tensor field $N_{\varphi}$ vanishes identically and
$(M,\varphi,\xi,\eta,g)$ is a Sasakian manifold.
\end{proof}

The expressions of $\kappa$ and $\mu$ in \eqref{costanti0} should be
compared with the example presented  by Boeckx in his local
classification of non-Sasakian contact metric
$(\kappa,\mu)$-manifolds with $I_M\leq -1$ (cf. $\S$ 4 of
\cite{Boeckx-00}). Therefore in some sense Theorem \ref{legendre1}
may be regarded also as a generalization of the Boeckx construction
for every value of the invariant $I_M$.

Furthermore, it should be remarked that the cases (I), (II) and
(III) of Theorem \ref{legendre1} correspond, respectively, to the
classes (I), (II) and (III) of Theorem \ref{classificazione}. This
is also clear by the computation of the invariant $I_M$. Indeed by
\eqref{costanti0} we get straightforwardly $I_M=\frac{a+b}{|a-b|}$,
so that, according to the signs of $a$ and $b$, $I_M$ can assume
values strictly greater than $1$, strictly lower than $-1$, or in
the interval $(-1,1)$. However an easy computation shows that
$I_M=\pm 1$ if and only if $a=0$ or $b=0$, that's impossible because
of the assumptions of Theorem \ref{legendre1}. Now we complete our
results by proving the following theorem concerning the remaining
classes (IV) and (V) of Theorem \ref{classificazione}.

\begin{theorem}\label{legendre2}
Let $(M,\eta)$ be a contact manifold endowed with a  bi-Legendrian
structure $({\mathcal F}_1,{\mathcal F}_2)$ such that
$\nabla^{bl}\Pi_{{\mathcal F}_1}=0$ (respectively,
$\nabla^{bl}\Pi_{{\mathcal F}_2}=0$). Assume that  ${\mathcal F}_1$
is positive definite (respectively, flat) and ${\mathcal F}_2$ is
flat (respectively, negative definite). Then for each $0<c\leq4$
(respectively, $-4\leq c<0$)  $(M,\eta)$ admits a
 compatible contact metric $(\kappa,\mu)$-structure, whose
associated bi-Legendrian structure is $({\mathcal F}_1,{\mathcal
F}_2)$, where
\begin{equation}\label{costants}
\kappa=1-\frac{c^2}{16}, \  \ \ \mu=2\left(1-\frac{c}{4}\right).
\end{equation}
\end{theorem}
\begin{proof}
Let us assume that  ${\mathcal F}_1$ is positive definite
 and ${\mathcal F}_2$ is flat. Then ${\mathcal F}_1$ is, in particular, non-degenerate and we
can consider the corresponding linear map $\Lambda_{{\mathcal
F}_1}:TM\longrightarrow T{\mathcal F}_1$ defined by
\eqref{lambda}. \ Since \ the \ operator \ $\Lambda_{{\mathcal
F}_1}$ \ is \ surjective \ and \ its \ kernel \  is \ $T{\mathcal
F}_1\oplus\mathbb{R}\xi$, \ we \ have  \ that $\Lambda_{{\mathcal
F}_1}|_{T{\mathcal F}_2}:T{\mathcal F}_2\longrightarrow T{\mathcal
F}_1$ is an isomorphism. Then for each $c\in (0,4]$ we define a
tensor field $\varphi$ of type $(1,1)$ by
\begin{equation}\label{strutture}
\varphi|_{T{\mathcal F}_1}:=\frac{1}{c}(\Lambda_{{\mathcal
F}_1}|_{T{\mathcal F}_2})^{-1}, \ \ \varphi|_{T{\mathcal F}_2}:=-c
\Lambda_{{\mathcal F}_1}|_{T{\mathcal F}_2}, \ \ \varphi\xi=0.
\end{equation}
Moreover we put
\begin{equation}\label{definizionemetrica1}
g|_{T{\mathcal F}_1\times T{\mathcal
F}_1}:=\frac{1}{c}\Pi_{{\mathcal F}_1}, \ \ g|_{T{\mathcal
F}_2\times T{\mathcal F}_2}:=c \overline{\Pi}_{{\mathcal
F}_1}|_{{T{\mathcal F}_2}\times{T{\mathcal F}_2}}, \ \
g:=\eta\otimes\eta \ \textrm{elsewhere}.
\end{equation}
Notice that $g$ defines a Riemannian metric since, by assumption,
${\mathcal F}_1$ is positive definite and $c>0$. We prove that in
fact $(\varphi,\xi,\eta,g)$ is a contact metric structure. Indeed we
have easily that $\varphi^2=-I+\eta\otimes\xi$. Next, \ for all \
$X,X'\in\Gamma(T{\mathcal F}_1)$ we have \ $g(\varphi X,\varphi
X')=c\Pi_{{\mathcal F}_1}(\Lambda_{{\mathcal F}_1}\varphi
X,\Lambda_{{\mathcal F}_1}\varphi X')=\frac{1}{c}\Pi_{{\mathcal
F}_1}(\Lambda_{{\mathcal F}_1} \Lambda_{{\mathcal F}_1}^{-1}
X,\Lambda_{{\mathcal F}_1} \Lambda_{{\mathcal F}_1}^{-1}
X')=\frac{1}{c}\Pi_{{\mathcal F}_1}(X,X')=g(X,X')$. In a similar way
one can prove that $g(\varphi Y,\varphi Y')=g(Y,Y')$ for all
$Y,Y'\in\Gamma(T{\mathcal F}_2)$. Moreover, the same arguments used
in the proof of Theorem \ref{legendre1} show that $g$ is an
associated metric, that is $d\eta=g(\cdot,\varphi\cdot)$. Thus
$(\varphi,\xi,\eta,g)$ is a contact metric structure. Notice that,
by construction, ${\mathcal F}_1$ and ${\mathcal F}_2$ are conjugate
Legendre foliations with respect to $(\varphi,\xi,\eta,g)$. Finally,
the definition of $g$ and the assumption $\nabla^{bl}\Pi_{{\mathcal
F}_1}=0$ imply that the bi-Legendrian connection is metric with
respect to $g$. Hence, by Lemma \ref{lemmarocky}, the tensor field
$\varphi$ is $\nabla^{bl}$-parallel and the operator $h$ preserves
the Legendre foliations ${\mathcal F}_1$ and ${\mathcal F}_2$. We
are now able to compute the explicit expression of $h$. For any
$X\in\Gamma(T{\mathcal F}_1)$ we have $2hX=[\xi,\varphi
X]-\varphi[\xi,X]=[\xi,\varphi X]_{T{\mathcal F}_1}+[\xi,\varphi
X]_{T{\mathcal F}_2}-\varphi([\xi,X]_{T{\mathcal
F}_1})-\varphi([\xi,X]_{T{\mathcal F}_2})$. The flatness of
${\mathcal F}_2$ yields $[\xi,\varphi X]_{T{\mathcal F}_1}=0$. Thus
\begin{equation}\label{equazione1}
2hX+\varphi([\xi,X]_{T{\mathcal F}_2})=[\xi,\varphi X]_{T{\mathcal
F}_2}-\varphi([\xi,X]_{T{\mathcal F}_1}).
\end{equation}
Since $h$ preserves the foliations, the right hand side of
\eqref{equazione1} is a section of both $T{\mathcal F}_1$ and
$T{\mathcal F}_2$, hence vanishes. Consequently, taking into account
that $\ker(\Lambda_{{\mathcal F}_1})=T{\mathcal
F}_1\oplus\mathbb{R}\xi$,
\begin{equation}\label{equazione2}
hX=-\frac{1}{2}\varphi([\xi,X]_{T{\mathcal
F}_2})=\frac{c}{2}\Lambda_{{\mathcal F}_1}([\xi,X]_{T{\mathcal
F}_2})=\frac{c}{2}\Lambda_{{\mathcal F}_1}([\xi,X])=\frac{c}{4}X.
\end{equation}
Moreover, let $Y$ be a section of $T{{\mathcal F}_2}$. As
$\varphi(T{\mathcal F}_1)=T{{\mathcal F}_2}$, $Y=\varphi X$ for some
$X\in\Gamma(T{\mathcal F}_1)$. Then, by \eqref{equazione2},
$hY=h\varphi X=-\varphi h X=-\frac{c}{4}\varphi X=-\frac{c}{4}Y$.
Thus the bi-Legendrian structure $({\mathcal F}_1,{\mathcal F}_2)$
coincides with that one determined by the eigendistributions of the
operator $h$. In particular this implies that $\nabla^{bl}h=0$.
Therefore all the conditions in Theorem \ref{principale0} are
verified and we conclude that $(M,\varphi,\xi,\eta,g)$ is a contact
metric $(\kappa,\mu)$-manifold such that ${\mathcal
D}(\lambda)={\mathcal F}_1$ and ${\mathcal D}(-\lambda)={\mathcal
F}_2$. Finally, comparing \eqref{definizionemetrica1} with
\eqref{invariante1} and taking into account that $\Pi_{{\mathcal
F}_2}=0$, we get $\kappa=1-\left(\frac{c}{4}\right)^2$ and
$\mu=2\left(1-\frac{c}{4}\right)$. The case when ${\mathcal F}_1$ is
flat and ${\mathcal F}_2$ is negative definite is analogous, the
only difference being to use $\Lambda_{{\mathcal F}_2}$ setting
\begin{gather*}
\varphi|_{T{\mathcal F}_1}:=-c\Lambda_{{\mathcal F}_2}|_{T{\mathcal
F}_1}, \ \ \varphi|_{T{\mathcal F}_2}:=\frac{1}{c}(\Lambda_{{\mathcal F}_2}|_{T{\mathcal F}_1})^{-1}, \ \ \varphi\xi=0, \\
g|_{T{\mathcal F}_1\times T{\mathcal
F}_1}:=c\overline{\Pi}_{{\mathcal F}_2}|_{T{\mathcal F}_1\times
T{\mathcal F}_1}, \  \ g|_{T{\mathcal F}_2\times T{\mathcal
F}_2}:=\frac{1}{c}\Pi_{{\mathcal F}_2}, \ g:=\eta\otimes\eta \ \
\textrm{elsewhere}.
\end{gather*}
where $c\in[-4,0)$. Arguing as in the previous case, one can find
that $(\varphi,\xi,\eta,g)$ is a contact metric
$(\kappa,\mu)$-structure, where $\kappa$ and $\mu$ are given by
\eqref{costants} and ${\mathcal D}(\lambda)={\mathcal F}_1$,
${\mathcal D}(-\lambda)={\mathcal F}_2$
\end{proof}

\begin{remark}
Notice that, as expected, by \eqref{costants}, we get $I_M=1$ if
$c>0$ and $I_M=-1$ if $c<0$. Furthermore, it should be remarked that
for no value of $c$ one can obtain a Sasakian structure, since
$\kappa=1$ if and only if $c=0$. Whereas, for $c=4$ one gets
$\kappa=\mu=0$, that is $R_{X Y}\xi=0$ for all $X,Y\in\Gamma(TM)$.
Such contact metric manifolds were deeply studied by Blair in
\cite{blair-1}.
\end{remark}

\begin{corollary}\label{sasaki1}
Let $(M,\varphi,\xi,\eta,g)$ be a non-Sasakian contact metric
$(\kappa,\mu)$-manifold. Then
\begin{enumerate}
  \item[(i)] if $I_M\neq \pm 1$, $(M,\eta)$ admits a family of compatible
contact metric $(\kappa_{a,b},\mu_{a,b})$-structures, where $a$ and
$b$ are real numbers such that $ab=(2-\mu)^2-4(1-\kappa)$;
  \item [(ii)] if $I_M=1$ (respectively, $I_M=-1$), $(M,\eta)$ admits a family of compatible contact
metric $(\kappa_{c},\mu_{c})$-structures, where $0<c\leq 4$
(respectively, $-4\leq c<0$).
\end{enumerate}
Furthermore, the above contact metric $(\kappa_{a,b},\mu_{a,b})$ and
$(\kappa_{c},\mu_{c})$-structures are of the same classification as
$(M,\varphi,\xi,\eta,g)$.
\end{corollary}
\begin{proof}
In order to prove the statements, it suffices to show that
$(M,\varphi,\xi,\eta,g)$ verifies all the hypotheses of Theorem
\ref{legendre1} for the case (i) and of Theorem \ref{legendre2} for
the case (ii).\\
(i) \ By \eqref{invariante1}--\eqref{invariante2} and by Theorem
\ref{principale0} we have immediately that
$\nabla^{bl}\Pi_{{\mathcal D}(\lambda)}=\nabla^{bl}\Pi_{{\mathcal
D}(-\lambda)}=0$, where, as usual, $\nabla^{bl}$ denotes the
bi-Legendrian connection associated to the bi-Legendrian structure
$({\mathcal D}(\lambda),{\mathcal D}(-\lambda))$ defined by the
eigendistributions of the operator $h$. Next, we compute the
explicit expression of the Libermann operators $\Lambda_{{\mathcal
D}(\lambda)}:TM\longrightarrow{\mathcal D}(\lambda)$ and
$\Lambda_{{\mathcal D}(-\lambda)}:TM\longrightarrow{\mathcal
D}(-\lambda)$. For any $X\in\Gamma({\mathcal D}(\lambda))$ and
$Y\in\Gamma({\mathcal D}(-\lambda))$ we have, by
\eqref{invariante1},
\begin{equation*}
\Pi_{{\mathcal D}(\lambda)}(\Lambda_{{\mathcal
D}(\lambda)}Y,X)=d\eta(Y,X)=g(Y,\varphi
X)=-\frac{1}{2\sqrt{1-\kappa}-\mu+2}g(\varphi Y,X),
\end{equation*}
from which it follows that
\begin{equation}\label{invariante3}
\Lambda_{{\mathcal D}(\lambda)}=\left\{
                                  \begin{array}{ll}
                                    0, & \hbox{on ${\mathcal D}(\lambda)\oplus\mathbb{R}\xi$,} \\
                                    \frac{1}{\mu-2-2\sqrt{1-\kappa}}\varphi, & \hbox{on ${\mathcal D}(-\lambda)$.}
                                  \end{array}
                                \right.
\end{equation}
Whereas, using \eqref{invariante2}, one can find
\begin{equation}\label{invariante4}
\Lambda_{{\mathcal D}(-\lambda)}=\left\{
                                  \begin{array}{ll}
                                    \frac{1}{\mu-2+2\sqrt{1-\kappa}}\varphi, & \hbox{on ${\mathcal D}(\lambda)$,} \\
                                    0, & \hbox{on ${\mathcal D}(-\lambda)\oplus\mathbb{R}\xi$.}
                                  \end{array}
                                \right.
\end{equation}
Notice that the denominators in \eqref{invariante3} and
\eqref{invariante4} are different from zero just because of the
assumption $I_{M}\neq\pm 1$. Next, for all $X,X'\in\Gamma({\mathcal
D}(-\lambda))$,
\begin{gather*}
\overline{\Pi}_{{\mathcal D}(-\lambda)}(X,X')=\Pi_{{\mathcal
D}(-\lambda)}(\Lambda_{{\mathcal D}(-\lambda)}X,\Lambda_{{\mathcal
D}(-\lambda)}X')=\frac{1}{-2\sqrt{1-\kappa}-\mu+2}g(\varphi
X,\varphi X')\\
=\frac{1}{-2\sqrt{1-\kappa}-\mu+2}g(X,X')=\frac{1}{(2-\mu)^2-4(1-\kappa)}\Pi_{{\mathcal
D}(\lambda)}(X,X').
\end{gather*}
Thus $\overline{\Pi}_{{\mathcal
D}(\lambda)}=((2-\mu)^2-4(1-\kappa))\overline{\Pi}_{{\mathcal
D}(-\lambda)}$ on ${\mathcal D}(\lambda)$ and in a similar manner
one can find that $\overline{\Pi}_{{\mathcal
D}(-\lambda)}=((2-\mu)^2-4(1-\kappa))\overline{\Pi}_{{\mathcal
D}(\lambda)}$ on ${\mathcal D}(-\lambda)$. We distinguish the cases
(I) $I_M>1$, (II) $-1<I_M<1$ and (III) $I_M<-1$. By Theorem
\ref{classificazione}, in the first case both ${\mathcal
D}(\lambda)$ and ${\mathcal D}(-\lambda)$ are positive definite, in
the second ${\mathcal D}(\lambda)$ is positive definite and
${\mathcal D}(-\lambda)$ negative definite and in the third one both
${\mathcal D}(\lambda)$ and ${\mathcal D}(-\lambda)$ are negative
definite. Then we take any two $a,b\in\mathbb{R}$ such that
$ab=(2-\mu)^2-4(1-\kappa)$ and $a>0$, $b>0$ in the case (I), $a>0$,
$b<0$ in the case (II) and $a<0$, $b<0$ in the case (III). Thus in
any case the hypotheses of Theorem \ref{legendre1} are verified and
so the structure $(\varphi_{a,b},\xi,\eta,g_{a,b})$ defined by
\eqref{costruzionemetrica} and \eqref{struttura} is a
contact metric $(\kappa_{a,b},\mu_{a,b})$-structure on $(M,\eta)$.\\
(ii) \ If $I_M=1$ then, by Theorem \ref{classificazione}, ${\mathcal
D}(\lambda)$ is positive definite and ${\mathcal D}(-\lambda)$ is
flat. Moreover, again by \eqref{invariante1} and Theorem
\ref{principale0} we have that $\nabla^{bl}\Pi_{{\mathcal
D}(\lambda)}=0$. Thus all the assumptions of Theorem \ref{legendre2}
are satisfied and it suffices to take any $c\in (0,4]$ for obtaining
a contact metric $(\kappa_{c},\mu_{c})$-structure given by
\eqref{strutture} and \eqref{definizionemetrica1}.
\end{proof}

\begin{corollary}\label{sasaki2}
Any contact metric $(\kappa,\mu)$-manifold $(M,\varphi,\xi,\eta,g)$
such that $|I_M|>1$ admits a compatible Sasakian structure.
\end{corollary}
\begin{proof}
If $(M,\varphi,\xi,\eta,g)$ is Sasakian then the assertion is
trivial, so we can assume that the structure $(\varphi,\xi,\eta,g)$
is non-Sasakian. Then we can apply Corollary \ref{sasaki1}. The
assumption $|I_M|>1$ implies by Theorem \ref{classificazione} that
the Legendre foliations ${\mathcal D}(\lambda)$ and ${\mathcal
D}(-\lambda)$ are either both positive definite or both negative
definite. So it is sufficient to take
$a=b=\sqrt{\left(1-\frac{\mu}{2}\right)^2-(1-\kappa)}$ in the case
of positive definiteness and
$a=b=-\sqrt{\left(1-\frac{\mu}{2}\right)^2-(1-\kappa)}$ in the case
of negative definiteness, for obtaining, by Theorem \ref{legendre1},
a Sasakian structure on $M$ compatible with the contact form $\eta$.
\end{proof}

\begin{corollary}
For each $1\leq p\leq 2n$, the $p$-th Betti number of a compact
contact metric $(\kappa,\mu)$-manifold $M^{2n+1}$ such that
$|I_{M^{2n+1}}|>1$, is even.
\end{corollary}
\begin{proof}
The assertion is a consequence of Corollary \ref{sasaki2} and the
results in \cite{blair-goldberg} and \cite{fujitani}.
\end{proof}

Finally, applying twice Theorem \ref{legendre1} and  Corollary
\ref{sasaki2} we get the following result.

\begin{corollary}
Let $(M,\eta)$ be a contact manifold endowed with two positive
definite or  negative definite Legendre foliations satisfying the
conditions (I) or (III) of Theorem \ref{legendre1}, respectively.
Then $(M,\eta)$ admits a compatible Sasakian structure.
\end{corollary}

We conclude  by recalling the definition of Tanaka-Webster parallel
space, recently introduced by Boeckx and Cho (\cite{Boeckx-08}). A
contact metric manifold is a \emph{Tanaka-Webster space} if its
generalized Tanaka-Webster torsion tensor $\hat{T}$ and its
curvature tensor $\hat{R}$ satisfy $\hat{\nabla} \hat{T}=0$ and
$\hat{\nabla} \hat{R}=0$, that is the Tanaka-Webster connection
$\hat{\nabla}$ is  invariant by parallelism (in the sense of
\cite{kobayashi1}). Boeckx and Cho have proven that a contact metric
manifold $M$ is a Tanaka-Webster parallel space if and only if $M$
is a Sasakian locally $\varphi$-symmetric space or a non-Sasakian
$(\kappa,2)$-space (\cite[Theorem 12]{Boeckx-08}). Thus, in
particular, we deduce the following corollaries of Theorem
\ref{legendre1} and of Corollary \ref{sasaki1}.

\begin{corollary}
Any non-Sasakian contact metric $(\kappa,\mu)$-manifold
$(M,\varphi,\xi,\eta,g)$ such that $|I_M|<1$ admits a compatible
Tanaka-Webster parallel structure.
\end{corollary}
\begin{proof}
The assumption $|I_M|<1$ implies by Theorem \ref{classificazione}
that the Legendre foliation ${\mathcal D}(\lambda)$ is positive
definite and ${\mathcal D}(-\lambda)$ is negative definite. \ So
it is sufficient to take
$a=-b=\sqrt{\left(1-\frac{\mu}{2}\right)^2-(1-\kappa)}$ for
obtaining, according to Theorem \ref{legendre1}, a compatible
contact metric $(\kappa_{a,b},\mu_{a,b})$-structure
$(\varphi_{a,b},\xi,\eta,g_{a,b}$ on $(M,\eta)$ such that
$\kappa=1-\frac{a^2}{4}$ and $\mu=2$. Thus, by applying the
aforementioned result by Boeckx and Cho, we conclude that
$(M,\varphi_{a,b},\xi,\eta,g_{a,b})$ is a Tanaka-Webster parallel
space.
\end{proof}

\begin{corollary}
Let $(M,\eta)$ be a contact manifold endowed with a positive
definite Legendre foliation ${\mathcal F}_1$ and negative definite
Legedre foliation ${\mathcal F}_1$ satisfying the condition (II)  of
Theorem \ref{legendre1}. Then $(M,\eta)$ admits a compatible
Tanaka-Webster parallel structure.
\end{corollary}

\end{document}